\documentclass[12pt,draft]{amsart}
\usepackage[all]{xy}
\usepackage{amsfonts}
\usepackage{amssymb}

\usepackage{color}

\newtheorem{theorem}{Theorem}[section]
\newtheorem{lemma}[theorem]{Lemma}
\newtheorem{prop}[theorem]{Proposition}
\newtheorem{cor}[theorem]{Corollary}
\newtheorem{que}[theorem]{Question}
\newtheorem{conj}[theorem]{Conjecture}
\theoremstyle{definition}
\newtheorem{dfn}[theorem]{Definition}
\newtheorem{rk}[theorem]{Remark}
\newtheorem{ex}[theorem]{Example}
\def\<{\langle}
\def\>{\rangle}

\def\M{{\mathcal M}}
\def\beqa{\begin{eqnarray*}}
\def\eeqa{\end{eqnarray*}}
\newcommand{\Z}{\mathbb{Z}}
\newcommand{\R}{\mathbb{R}}
\newcommand{\N}{\mathbb{N}}

\newcommand{\C}{\mathbb{C}}
\newcommand{\nc}{\newcommand}
\nc{\dmo}{\DeclareMathOperator}
\dmo{\Teich}{Teich} 
\dmo{\spec}{spec}
\dmo\Mod{Mod}
\def\Tr{\rm Tr\ } 
 
\newcommand{\de}{\mathrm{d}}
\newcommand{\De}{\mathrm{D}}
\renewcommand{\to}{\rightarrow}
\newcommand{\To}{\longrightarrow}

\newcommand{\inclusion}{\hookrightarrow}
\newcommand{\p}{\partial}
\newcommand{\Of}{\Omega_{\phi}}
\newcommand{\ophi}{\omega_{\phi}}
\DeclareMathOperator{\id}{id}

\DeclareMathOperator{\im}{im}

\DeclareMathOperator{\ind}{index}

\DeclareMathOperator{\vol}{vol}
\DeclareMathOperator{\area}{area}

\DeclareMathOperator{\fix}{Fix}

\DeclareMathOperator{\symp}{Symp}
\DeclareMathOperator{\diff}{Diff}

\DeclareMathOperator{\sign}{sign}
\DeclareMathOperator{\flux}{Flux}
\DeclareMathOperator{\grow}{Growth}

\def\be{\begin{enumerate}}
\def\ee{\end{enumerate}}
\def\ii{\item}
\def\bi{\begin{itemize}}
\def\ei{\end{itemize}}

\def\Im{\mathop{\rm Im}\nolimits}

\def\tr{\mathop{\rm tr}\nolimits}

\def\Tr{\operatorname{Tr}}

\def\Ind{\operatorname{Ind}}
\def\ind{\operatorname{ind}}

\newcommand{\HF}{HF_*}
\newcommand{\CF}{CF_*}

\def\Fix{\operatorname{Fix}}

\def\RR{\mathbb{R}}
\def\N{{\mathbb N}}

\begin{document}

\title[The growth rate of  Floer homology and symplectic zeta function]
{The growth rate of Floer homology and symplectic zeta function}

\author[Alexander Fel'shtyn]{Alexander Fel'shtyn}
\address{Instytut Matematyki, Uniwersytet Szczecinski,
ul. Wielkopolska 15, 70-451 Szczecin, Poland  and Institute for Advanced Study, Einstein Drive, Princeton, NJ 08540 USA}
\email{felshtyn@gmail.com, felshtyn@ias.edu, fels@wmf.univ.szczecin.pl}

\begin{abstract}

The main theme of this paper is to study  for a  symplectomorphism  of a compact  surface,
the asymptotic invariant which is defined to be the growth rate of
the sequence  of the total dimensions of symplectic Floer homologies
of the iterates of the symplectomorphism. We prove that the asymptotic invariant coincides  with asymptotic Nielsen number and with asymptotic absolute Lefschetz number. We also show that the asymptotic invariant coincides   with  the largest  dilatation of the pseudo-Anosov components of the symplectomorphism and its   logarithm coincides with  the  topological entropy.  This  implies  that  symplectic zeta function  has a positive  radius of convergence. This also establishes a connection between Floer homology and geometry of 3-manifolds.

\end{abstract}

\maketitle

\tableofcontents

\section{Introduction}

The main theme of this paper is to study  for a  symplectomorphism $\phi: M \rightarrow M$ in given mapping class $g$ of a compact  surface $M$,
a asymptotic invariant $F_{\infty}(g)$, introduced in \cite{f},  which is defined to be the growth rate of
the sequence $\dim\HF(\phi^n)$ of the total dimensions of symplectic Floer homologies
of the iterates of $\phi$. We prove  a  conjecture from \cite{f} which suggests that   the asymptotic invariant coincides  with asymptotic Nielsen number and   with  the largest  dilatation of the pseudo-Anosov components of $g$ and its   logarithm coincides with topological entropy.
This  establishes a connection between Floer homology and geometry of 3-manifolds.
 The asymptotic invariant  also  provides  the  radius of convergence 
of the symplectic zeta function $$F_g(t)= F_\phi(t) = 
 \exp\left(\sum_{n=1}^\infty \frac{\dim\HF(\phi^n)}{n} t^n \right)$$.  We  show that the symplectic zeta function   has a positive  radius of convergence
which admits exact algebraic estimation via Reidemeister trace formula.

Our main results are  the  following  
\begin{theorem}\label{th:pa1}  If $\phi$ is  any  symplectomorphism with nondegenerate fixed points in
given pseudo-Anosov mapping class $ g$  with dilatation  $\lambda >1$ of surfase $M$
of genus $\geq 2$. Then
 $$ F^{\infty}(g):=\grow(\dim\HF(\phi^n))=\lambda=\exp(h(\psi))= L^{\infty}(\psi)=N^{\infty}(\psi) $$
where $\psi$ is a canonical singular  pseudo-Anosov representative of $g$,  $h(\psi)$ is the topological entropy and $L^{\infty}(\psi)$ and $N^{\infty}(\psi)$ are asymptotic(absolute) Lefshetz number
and asymptotic Nielsen number.
\end{theorem}

\begin{theorem}\label{th:main} Let $\bar{\phi}$ be a perturbed standard form map $\phi$ in a reducible
mapping class $g$ of compact surface  of genus $\geq 2$ and  $\lambda$ is the largest  dilatation  of the  pseudo-Anosov components( $\lambda=1$ if there is no pseudo-Anosov components). Then
$$ F^{\infty}(g):=\grow(\dim\HF((\bar{\phi})^n))=\lambda=\exp(h(\psi))= L^{\infty}(\psi)=N^{\infty}(\psi)$$
 where $\psi$ is canonical representative of mapping class $g$.
\end{theorem}

\begin{rk}

The genus one case follows from \cite{f} and Pozniak's thesis\cite{p}.

\end{rk}
\begin{theorem}
 Let $\bar{\phi}$ be a perturbed standard form map $\phi$ in a reducible
mapping class $g$ of compact surface  of genus $\geq 2$ and  $\lambda$ is the largest  dilatation  of the  pseudo-Anosov components( $\lambda=1$ if there is no pseudo-Anosov components). Then
 the symplectic  zeta function $F_g(t)= F_{\bar{\phi}(t)}$ has positive radius of convergence  $R=\frac{1}{\lambda}$, where  $\lambda$ is the largest  dilatation  of the  pseudo-Anosov components( $\lambda=1$ if there is no pseudo-Anosov components).
\end{theorem}

Although the exact evaluation of the asymptotic invariant would be desirable, in general,  its estimation is a 
more realistic goal and as we shall show, one that is sufficient for some applications.

We suggested in \cite{f} that the asymptotic invariant potentially  may be important for the applications.
Recent paper of Ivan Smith \cite{Sm} gives    application of the asymptotic invariant to the 
important question of faithfulness of a representation of extended mapping class group  via considerations motivated by Homological Mirror Symmetry.

{\bf Acknowledgments.} I would like to thank Andrew Cotton-Clay, Yasha Eliashberg,  Sam Lewallen,  Ivan Smith and Andras Stipsicz
for helpful  discussions.
I am very grateful to the Stanford University, Mathematical Science  Research Institute, Berkeley,  Max-Planck-Institute for
Mathematics, Bonn, Institute des Hautes Etudes Scientifiques, Bures-sur-Yvette and Institute for Advanced Studies, Princeton for their kind hospitality and  support during the preparation of this paper.  

\section{Preliminaries}

\subsection{ Symplectic Floer homology}

\subsubsection{Review of monotonicity and weak monotonicity}

In this section we discuss the notion of monotonicity and weak monotonicity  as defined in \cite{S,G,c1}.
Monotonicity plays important role for  Floer homology in two
dimensions.
Throughout this article, $M$ denotes a compact  connected and oriented
2-manifold of genus $\geq2$.   Pick an everywhere positive two-form $\omega$
on $M$.

Let $\phi\in\symp(M,\omega)$, the group of  symplectic automorphisms
 of the two-dimensional symplectic
manifold $(M,\omega)$ (when $M$ has boundary we consider the group of orientation-preserving diffeomorphisms of $M$ with
no fixed points on the boundary).
The mapping torus of $\phi$,  $T_\phi = \R\times M/(t+1,x)\sim(t,\phi(x)),$
is a 3-manifold fibered over $S^1=\R/\Z$.
There are two natural second
cohomology classes on $T_\phi$, denoted by $[\ophi]$ and $c_\phi$. The first one is
represented by the closed two-form $\ophi$ which is induced from the pullback
of $\omega$ to $\R\times M$. The second is the Euler class of the
vector bundle
$
V_\phi =
\R\times T M/(t+1,\xi_x)\sim(t,\de\phi_x\xi_x),
$
which is of rank 2 and inherits an orientation from $TM$.

Symplectomorphism $\phi\in\symp(M,\omega)$ is called {\bf monotone}, if
$
[\omega_\phi] = (\area_\omega(M)/\chi(M))\cdot c_\phi
$
in $H^2(T_\phi;\R)$; throughout this article
$\symp^m(M,\omega)$ denotes the set of
monotone symplectomorphisms.

Now $H^2(T_\phi;\R)$ fits into the following short exact sequence \cite{S,G}
\begin{equation}\label{eq:cohomology}
0 \To  \frac{H^1(M;\R)}{\im(\id-\phi^*)}
\stackrel{d}{\To} H^2(T_\phi;\R)
\stackrel{r^*}{\To} H^2(M;\R),
\To 0.
\end{equation}
where the map $r^*$  is restriction to the fiber.
The map $d$ is defined as follows.
Let $\rho:I\to\R$ be a smooth function which vanishes near $0$ and $1$ and
satisfies $\int_0^1\!\rho\,\de t=1$.
If $\theta$ is a closed 1-form on $M$, then
$\rho\cdot\theta\wedge\de t$ defines a closed 2-form on $T_\phi$; indeed
$
d[\theta] = [\rho\cdot\theta\wedge\de t].
$
The map $r:M\inclusion T_\phi$ assigns to each $x\in M$ the
equivalence class of $(1/2,x)$.
Note, that $r^*\ophi=\omega$ and $r^*c_\phi$ is the Euler class of
$TM$.
Hence, by \eqref{eq:cohomology}, there exists a unique class
$m(\phi)\in H^1(M;\R)/\im(\id-\phi^*)$ satisfying
$
d\,m(\phi) = [\ophi]-(\area_\omega(M)/\chi(M))\cdot c_\phi,
$
where $\chi(M)$ denotes the Euler characteristic of $M$.
Therefore, $\phi$ is monotone if and only if $m(\phi)=0$.
 
Because $c_\phi$ controls the index, or expected dimension, of moduli
spaces of holomorphic curves under change of homology class and $\omega_\phi$
controls their energy under change of homology class, the monotonicity condition
ensures that the energy is constant on the index one components of the moduli space,
which implies compactness and, as a  corollary, finite count  in a  differential 
of the Floer complex.

We recall the fundamental properties of $\symp^m(M,\omega)$ from \cite{S,G}.
Let  $\diff^+(M)$ denotes  the group of orientation preserving
diffeomorphisms of $M$.
\smallskip\\
(Identity) $\id_M \in \symp^m(M,\omega)$.
\smallskip\\
(Naturality)\label{page:natur}
If $\phi\in\symp^m(M,\omega),\psi\in\diff^+(M)$, then
$\psi^{-1}\phi\psi\in\symp^m(M,\psi^*\omega)$.
\smallskip\\
(Isotopy)
Let $(\psi_t)_{t\in I}$ be an isotopy in $\symp(M,\omega)$, i.e. a smooth
path with $\psi_0=\id$.
Then
$
m(\phi\circ\psi_1)=m(\phi)+[\flux(\psi_t)_{t\in I}]
$
in $H^1(M;\R)/\im(\id-\phi^*)$; see \cite[Lemma 6]{S}.
For the definition of the flux homomorphism see \cite{MS}.
\smallskip\\
(Inclusion)
The inclusion $\symp^m(M,\omega)\inclusion\diff^+(M)$ is a homotopy
equivalence. In particular $\symp^m(M,\omega)$ is path connected.
\smallskip\\
(Floer homology)
To every $\phi\in\symp^m(M,\omega)$ symplectic Floer homology
theory assigns a $\Z_2$-graded vector space $\HF(\phi)$ over $\Z_2$, with an
additional multiplicative structure, called the quantum cap product,
$
H^*(M;\Z_2)\otimes\HF(\phi)\To\HF(\phi).
$
For $\phi=\id_M$ the symplectic Floer homology $\HF(\id_M)$ are  canonically isomorphic to ordinary homology  $H_*(M;\Z_2)$ and quantum cap product agrees with the ordinary cap product.
Each $\psi\in\diff^+(M)$ induces an isomorphism
$\HF(\phi)\cong\HF(\psi^{-1}\phi\psi)$ of $H^*(M;\Z_2)$-modules.
\smallskip\\
(Invariance)
If $\phi,\phi'\in\symp^m(M,\omega)$ are isotopic, then
$\HF(\phi)$ and $\HF(\phi')$ are naturally isomorphic as
$H^*(M;\Z_2)$-modules.
This is proven in \cite[Page 7]{S}. Note that every Hamiltonian perturbation
of $\phi$ (see \cite{ds}) is also in $\symp^m(M,\omega)$.
\smallskip\\
Now let $g$ be a mapping class of $M$, i.e. an isotopy class of $\diff^+(M)$.
Pick an area form $\omega$ and a
representative $\phi\in\symp^m(M,\omega)$ of $g$.
 $\HF(\phi)$ is an invariant as $\phi$ is deformed through monotone symplectomorphisms.
These imply that we have a symplectic Floer homology invariant $\HF(g)$ canonically assigned to each
mapping class  $g$ given by $\HF(\phi)$ for any monotone symplectomorphism $\phi$. Note that $\HF(g)$ is independent of the choice of an area
form $\omega$ by Moser's isotopy theorem \cite{Mo} and naturality of Floer homology.

We give now,  following A. Cotton-Clay \cite{c1},  a notion of \emph{weak monotonicity}  such that $HF_*(\phi)$ is well-defined for and invariant among weakly monotone simplectomorphisms. Monotonicity implies weak monotonicity, and so $HF_*(g) = HF_*(\phi)$ for any weakly monotone $\phi$ in mapping class $g$. The properties of weak monotone symplectomorphism
of surface  play a crucial role in the  computation of  Floer homology for pseudo-Anosov and  reducible mapping classes(see \cite{c1}).

A symplectomorphism  $\phi: M \rightarrow M$ is \emph{\bf weakly monotone}
 if $[\omega_\phi]$ vanishes on\\
 $\ker(c_\phi|_{T(T_\phi)})$, where $T(T_\phi) \subset H_2(M_\phi;\R)$ is generated by tori $T$ such that $\pi|_T: T \rightarrow S^1$ is a fibration with fiber $S^1$, where the map $\pi: T_\phi \rightarrow S^1$ is the projection.
 Throughout this article
$\symp^{wm}(M,\omega)$ denotes the set of
 weakly monotone symplectomorphisms.

\subsubsection{Floer homology}

Let $\phi\in\symp(M,\omega)$.There are two ways of constructing Floer
homology detecting its fixed points, $\Fix(\phi)$. Firstly, the graph of $\phi$
is a Lagrangian submanifold of $M\times M,(-\omega)\times\omega)$ and its fixed points correspond to the intersection points of graph($\phi$) with the
diagonal $\Delta=\{(x,x)\in M\times M\}$. Thus we have the Floer homology of the Lagrangian intersection $\HF(M\times M,\Delta, graph (\phi))$.
This intersection is transversal if the fixed points of $\phi$ are nondegenerate, i.e. if 1 is not an eigenvalue of $d\phi(x)$, for $x\in\Fix(\phi)$.
The second approach was mentioned by Floer in \cite{Floer1} and presented with
details by Dostoglou and Salamon  in \cite{ds}.We follow here  Seidel's approach \cite{S} which, comparable with \cite {ds}, uses a larger
class of perturbations, but such that the perturbed action form is still
cohomologous to the unperturbed. As a consequence, the usual invariance of
Floer homology under Hamiltonian isotopies is extended to the stronger
property stated above.
Let now    $\phi$ is  monotone or weakly monotone.
Firstly, we give  the definition of $\HF(\phi)$  in the special case
where all the fixed points of $\phi$ are non-degenerate, i.e. for all $y\in\Fix(\phi)$, $\det(\id-\de\phi_y)\ne0$, and then
following Seidels approach  \cite{S} we consider general case when $\phi$ has degenerate fixed points.
 Let $\Of = \{ y \in C^{\infty}(\R,M)\,|\, y(t) = \phi(y(t+1)) \}$ be the twisted free loop space, which is also
the space of sections of $T_\phi \rightarrow S^1$. The action form is the
closed one-form $\alpha_\phi$ on $\Of$ defined by
$$
\alpha_\phi(y) Y = \int_0^1 \omega(dy/dt,Y(t))\,dt.
$$
where $y\in\Of$ and $Y\in T_y\Of$, i.e. $Y(t)\in T_{y(t)}M$ and
$Y(t)=\de\phi_{y(t+1)}Y(t+1)$ for all $t\in\R$.

The tangent bundle of any symplectic manifold admits an almost complex structure $ J:TM\To TM$ which is compatible with $\omega$ in sense that $(v,w)=\omega(v,Jw)$ defines a Riemannian metric.
 Let $J=(J_t)_{t \in \R}$
be a smooth path of $\omega$-compatible almost  complex structures on
$M$ such that $J_{t+1}=\phi^*J_t$.
If $Y,Y'\in T_y\Of$, then
$\int_0^1\omega(Y'(t),J_t Y(t))\de t$ defines a metric
on the loop space $\Of$. So  the critical points of $\alpha_\omega$ are the constant paths in  $\Of$ and  hence the fixed points of $\phi$. The negative gradient lines of $\alpha_\omega$ with respect to the
metric above  are solutions of the partial differential equations
with boundary conditions
\begin{equation}\label{eq:corbit}
\left\{\begin{array}{l}
u(s,t) = \phi(u(s,t+1)), \\
\p_s u + J_t(u)\p_t u = 0, \\
\lim_{s\to\pm\infty}u(s,t) \in \Fix(\phi)
\end{array}\right.
\end{equation}
These are exactly  Gromov's pseudoholomorphic  curves \cite{Gromov}.

For $y^\pm\in\Fix(\phi)$, let $\M(y^-,y^+;J,\phi)$ denote the space of smooth maps $u:\R^2\to M$ which satisfy the  equations \eqref{eq:corbit}.
Now to every $u\in\M(y^-,y^+;J,\phi)$ we  associate a Fredholm operator
$\De_u$ which linearizes (\ref{eq:corbit}) in suitable Sobolev spaces. The
index of this operator is given by the so called Maslov index $\mu(u)$,
which satisfies $\mu(u)=\deg(y^+)-\deg(y^-)\text{ mod }2$, where  $(-1)^{\deg y}=\sign(\det(\id-\de\phi_y))$. We have no bubbling, since for surface
$\pi_2(M)=0$. For a generic
$J$, every $u\in\M(y^-,y^+;J,\phi)$ is regular, meaning that $\De_u$ is onto.
Hence, by the implicit function theorem, $\M_k(y^-,y^+;J,\phi)$ is
a smooth $k$-dimensional manifold and  is  the
subset of those $u\in\M(y^-,y^+;J,\phi)$ with $\mu(u)=k\in\Z$.
Translation of the $s$-variable defines a free $\R$-action on 1-dimensional
manifold $\M_1(y^-,y^+;J,\phi)$ and hence the quotient is a discrete set of points. The energy of a map $u:\R^2\to M$ is given by $
E(u) = \int_{\R}\int_0^1 \omega\big(\p_tu(s,t),J_t\p_tu(s,t)\big)\,\de t\de s$  for all $y\in\Fix(\phi)$.  P.Seidel and A. Cotton-Clay have proved  in \cite{S} and \cite{c1} that if $\phi$ is monotone or weakly monotone, then the energy is constant on each $\M_k(y^-,y^+;J,\phi)$.
Since all fixed points of $\phi$ are nondegenerate the set  $\Fix(\phi)$ is a finite set and the
$\Z_2$-vector space  $
\CF(\phi) := \Z_2^{\#\Fix(\phi)}$
admits a $\Z_2$-grading with $(-1)^{\deg y}=\sign(\det(\id-\de\phi_y))$,
for all $y\in \Fix(\phi)$.
The boundedness of the energy  $E(u)$   for monotone or weakly monotone   $\phi$  implies that   the  0-dimensional  quotients   $\M_1(y_-,y_+,J,\phi)/\R$   are actually finite sets. Denoting by  $n(y_-,y_+)$  the number of
points mod 2 in each of them, one defines a differential  $\partial_{J}:
CF_*(\phi) \rightarrow CF_{* + 1}(\phi)$  by $\partial_{J}y_- =
\sum_{y_+} n(y_-,y_+) {y_+}$.  Due to gluing theorem  this Floer boundary operator satisfies  $\partial_{J} \circ
\partial_{J} = 0$.  For gluing  theorem to hold one needs again the  boundedness of the energy $E(u)$ .  It follows that  $ (\CF(\phi),\partial_{J})$  is a chain complex  and its homology is by definition the Floer homology of  $\phi$   denoted $HF_*(\phi)$. It  is independent of $J$ and is an invariant of $\phi$.

If $\phi$ has degenerate fixed points one needs to perturb equations
\eqref{eq:corbit} in order to define the Floer homology. Equivalently, one
could say that the action form needs to be perturbed.
 The necessary analysis is  given in \cite{S}, it  is essentially  the same as in the slightly different
situations considered in \cite{ds}. But  Seidel's approach also differs from the usual one in \cite{ds}. He uses a larger
class of perturbations, but such that the perturbed action form is still
cohomologous to the unperturbed.

\subsection{Nielsen classes  and Reidemeister trace}

Before discussing the   results of the paper, we briefly describe the few
basic notions of Nielsen fixed point theory which will be used.
We assume  $X$ to be a connected, compact
polyhedron and $f:X\rightarrow X$ to be a continuous map.
Let $p:\tilde{X}\rightarrow X$ be the universal cover of $X$
and $\tilde{f}:\tilde{X}\rightarrow \tilde{X}$ a lifting
of $f$, i.e. $p\circ\tilde{f}=f\circ p$.
Two liftings $\tilde{f}$ and $\tilde{f}^\prime$ are called
{\sl conjugate} if there is a $\gamma\in\Gamma\cong\pi_1(X)$
such that $\tilde{f}^\prime = \gamma\circ\tilde{f}\circ\gamma^{-1}$.
The subset $p(Fix(\tilde{f}))\subset Fix(f)$ is called
{\sl the fixed point class of $f$ determined by the lifting class $[\tilde{f}]$}.Two fixed points $x_0$ and $x_1$ of $f$ belong to the same fixed point class iff  there is a path $c$ from $x_0$ to $x_1$ such that $c \cong f\circ c $ (homotopy relative endpoints). This fact can be considered as an equivalent definition of a non-empty fixed point class.
 Every map $f$  has only finitely many non-empty fixed point classes, each a compact  subset of $X$.
A fixed point class is called {\sl essential} if its index is nonzero.
The number of essential fixed point classes is called the {\sl Nielsen number}
of $f$, denoted by $N(f)$.The Nielsen number is always finite.
$R(f)$ and $N(f)$ are homotopy invariants.
In the category of compact, connected polyhedra, the Nielsen number
of a map is, apart from certain exceptional cases,
 equal to the least number of fixed points
 of maps with the same homotopy type as $f$.

Let $f:X\rightarrow X$ be given, and let a
specific lifting $\tilde{f}:\tilde{X}\rightarrow\tilde{X}$ be chosen
as reference.
Let $\Gamma$ be the group of
covering translations of $\tilde{X}$ over $X$.
Then every lifting of $f$ can be written uniquely
as $\alpha\circ \tilde{f}$, with $\alpha\in\Gamma$.
So elements of $\Gamma$ serve as coordinates of
liftings with respect to the reference $\tilde{f}$.
Now for every $\alpha\in\Gamma$ the composition $\tilde{f}\circ\alpha$
is a lifting of $f$ so there is a unique $\alpha^\prime\in\Gamma$
such that $\alpha^\prime\circ\tilde{f}=\tilde{f}\circ\alpha$.
This correspondence $\alpha\rightarrow\alpha^\prime$ is determined by
the reference $\tilde{f}$, and is obviously a homomorphism.
The endomorphism $\tilde{f}_*:\Gamma\rightarrow\Gamma$ determined
by the lifting $\tilde{f}$ of $f$ is defined by
$
  \tilde{f}_*(\alpha)\circ\tilde{f} = \tilde{f}\circ\alpha.
$
It is well known that $\Gamma\cong\pi_1(X)$.
We shall identify $\pi=\pi_1(X,x_0)$ and $\Gamma$ in the usual  way.

We have seen that $\alpha \in \pi$ can be considered as the coordinate of the
lifting $\alpha \circ \tilde f$. We can  tell the conjugacy of two liftings from their coordinates:
$[\alpha \circ\tilde f]=[\alpha^\prime \circ\tilde f] $  iff there is $\gamma  \in \pi$ such that
$\alpha^\prime=\gamma \alpha \tilde{f}_* (\gamma^{-1})$.

So we have the  Reidemeister bijection:
 Lifting classes of $f$ are in 1-1 correspondence with $\tilde{f}_*$-conjugacy classes in group  $\pi$,
 the lifting class $[\alpha\circ\tilde f]$ corresponds to the $\tilde{f}_*$-cojugacy class of $\alpha$.

   By an abuse of language, we  say that the fixed point class $p( \fix{\alpha\circ\tilde f})$,
which is labeled with the lifting class $[\alpha\circ\tilde f]$,corresponds to the $\tilde{f}_*$-conjugacy class of $\alpha$. Thus the $\tilde{f}_*$-conjugacy classes in $\pi$ serve
as coordinates for the fixed point classes of $f$, once a reference lifting $\tilde f$ is chosen.

\subsubsection{Reidemeister trace}\label{section: Reidemeister}

 The results of this section are well known(see \cite{j},\cite{FelshB,fh2}).We shall use this results later in section  to estimate the radius of convergence of the symplectic  zeta function. 
 The fundamental group  $\pi=\pi_1(X,x_0)$ splits into $\tilde{f}_*$-conjugacy classes.Let $\pi_f$ denote the set of $\tilde{f}_*$-conjugacy classes,and $\Z\pi_f$ denote the Abelian group freely generated by $\pi_f$ .We will use the bracket notation $a\to [a]$ for both projections $\pi\to \pi_f$ and  $\Z\pi\to \Z\pi_f$. 
Let $x$ be a fixed point of $f$.Take a path $c$ from $x_0$ to $x$.The  $\tilde{f}_*$-conjugacy class in $\pi$ of the loop $c\cdot (f\circ c)^{-1}$,which is evidently independent of the choice of $c$, is called the coordinate of $x$.Two fixed points are in the same fixed point class $F$ iff they have the same coordinates.This
$\tilde{f}_*$-conjugacy class is thus called the coordinate of the fixed point class $F$ and denoted $cd_{\pi}(F,f)$ (compare with  description in section 2). 
 The generalized Lefschetz number  or the Reidemeister trace \cite{j} is defined as
\begin{equation} 
L_{\pi}(f):=\sum_{F}\ind (F,f)\cdot cd_{\pi}(F,f)  \in  \Z\pi_f , 
\end{equation} 
the summation being over all essential fixed point classes $F$of $f$.The Nielsen number $N(f)$  
is the number of non-zero terms in $L_{\pi}(f)$,and the indices of the essential fixed point classes  
appear as the coefficients in $L_{\pi}(f)$.This invariant used to be called the Reidemeister trace  
because it can be computed as an alternating sum of traces on the chain level   
as follows \cite{j} . 
Assume that $X$ is a finite cell complex and  $f:X\to X$ is a cellular map. 
A cellular decomposition ${e_j^d}$ of $X$ lifts to a $\pi$-invariant cellular structure
 on the universal covering $\tilde X$.Choose an arbitrary lift  ${\tilde{e}_j^d}$ for each ${e_j^d}$ . 
They constitute a free  $\Z\pi$-basis for the cellular chain complex of $\tilde{X}$. 
The lift $\tilde{f}$ of $f$ is also a cellular map.In every dimension $d$, the cellular chain map $\tilde{f}$  
 gives rise to a $\Z\pi$-matrix $ \tilde{F}_d $ with respect to the above basis,i.e  
$\tilde{F}_d=(a_{ij})$ if $ \tilde{f}(\tilde{e}_i^d)=\sum_{j}a_{ij}\tilde{e}_j^d $,where
 $ a_{ij}\in \Z\pi $.Then we have the Reidemeister trace formula
  
\begin{equation} 
L_{\pi}(f)=\sum_{d}(-1)^d[\Tr \tilde{F}_d] \in \Z\pi_f . 
\end{equation}

Now we describe alternative approach to the Reidemeister trace formula proposed  by Jiang \cite{j}. This approach is useful when we study the periodic points of $f$,i.e. the fixed points of the iterates of $f$.  
  
 The mapping torus $T_f$ of $f:X\rightarrow X$ is the space obtained from $X\times [o,\infty )$ by identifying $(x,s+1)$ with $(f(x),s)$ for all $x\in X,s\in [0 ,\infty )$.On $T_f$ there is a natural semi-flow $\phi :T_f\times [0,\infty )\rightarrow T_f, \phi_t(x,s)=(x,s+t)$ for all $t\geq 0$.Then the map  $f:X\rightarrow X$ is the return map of the semi-flow $\phi $.A point $x\in X$ and a positive number $\tau >0$ determine the orbit curve $\phi _{(x,\tau )}:={\phi_t(x)}_{0\leq t \leq \tau}$ in $T_f$. 
Take the base point $x_0$ of $X$ as the base point of $T_f$.It is known that the fundamental group $H:=\pi_ 1(T_f,x_0)$ is obtained from $\pi $ by adding a new generator $z$ and adding the relations $z^{-1}gz=\tilde f_*(g)$ for all $g\in \pi =\pi _1(X,x_0)$.Let  $H _c$ denote the set of conjugacy classes in $H $. Let $\Z H $ be the integral group ring of $H $, and let $\Z H_c $ be the free Abelian group with basis $ H _c $.We again use the bracket notation $ a\rightarrow [a] $  for both projections $H \rightarrow H _c $ and $ \Z H \rightarrow \Z H _c $.  If  $F^n$ is a fixed point class  of $f^n$, then  
$f(F^n)$ is also fixed point class of $f^n$ and
 $\ind (f(F^n),f^n)=\ind (F^n,f^n)$. 
Thus $f$ acts as an index-preserving permutation among fixed point classes of $f^n$.By definition, an $n$-orbit class $O^n$  of $f$ to be the union of elements of an orbit of this action.In other words, two points $x,x'\in \Fix (f^n)$ are said to be in the same $n$-orbit class of $f$ if and only if some $f^i(x)$ and some $f^j(x')$ are in the same fixed point class of $f^n$.The set $\Fix (f^n)$ splits into a disjoint union of $n$-orbits classes.Point $x$ is a fixed point of $f^n$   
or a periodic point of period $n$ if and only if orbit curve  $\phi _{(x,n)}$ is a closed curve.  
The free homotopy class of the closed curve $\phi _{(x,n)}$ will be called the $H$ -coordinate  
of point $x$,written $cd_{H }(x,n)=[\phi _{(x,n)}]\in H _c$.It follows that periodic points $x$  
of period $n$ and $x'$ of period $n'$ have the same $H $-coordinate if and only if $n=n'$ 
 and $x$,$x'$ belong to the same $n$-orbits class of $f$.  
Thus it is possible equivalently define $x,x'\in \Fix (f^n) $ to be in the  
same $n$-orbit class if and only if they have the same $H-$coordinate.  
  Jiang \cite{j} has considered generalized Lefschetz number with respect to $H $ 
\begin{equation} 
L_{H }(f^n):= \sum_{O^n}\ind (O^n,f^n)\cdot cd_{H }(O^n) \in \Z H _c, 
\end{equation} 
and proved following trace formula: 
\begin{equation} 
L_{H }(f^n)=\sum_{d}(-1)^d[\Tr (z\tilde{F}_d)^n] \in \Z H_c, 
\end{equation} 
where $\tilde{F}_d$ be $\Z \pi$-matrices defined in (16) and $z\tilde{F}_d$ is regarded as a $\Z H $-matrix. 

\subsubsection{Twisted Lefschetz numbers and  twisted Lefschetz zeta function} \label{section: Lefschetz}

Let $R$ be a commutative ring with unity. Let $GL_n(R)$ be the group of invertible
$n\times n $ matrices in $R$, and $ M_{n\times n}(R) $ be the algebra of $n\times n $ matrices in $R$.
Suppose  a representation $\rho: H \rightarrow  GL_n(R)$ is given. It extends to a representation
$\rho: \Z H \rightarrow M_{n\times n}(R) $. Following  Jiang \cite{j} we define $\rho$-twisted 
Lefschetz  number 
\begin{equation} 
L_{\rho}(f^n):= \Tr (L_{H }(f^n))^{\rho}=\sum_{O^n}\ind (O^n,f^n)\cdot \Tr(cd_{H }(O^n))^{\rho} \in R, 
\end{equation} 
where $h^{\rho}$ is $\rho$-image of $h\in \Z H$.
It has the trace formula(see \cite{j})
\begin{equation} 
L_{\rho}(f^n)=\sum_{d}(-1)^d\Tr( (z\tilde{F}_d)^{\rho})^n \in R,
\end{equation} 
where for a $\Z H$-matrix $A$, its $\rho$-image $A^{\rho}$ means the block matrix
obtained from $A$ by replacing each element $a_{ij}$ with $n\times n$ $R$-matrix $a_{ij}^{\rho}$.
Twisted Lefschetz zeta function is defined as formal power series
$$L^f_{\rho}(t):= 
 \exp\left(\sum_{n=1}^\infty \frac{L_{\rho}(f^n)}{n} t^n \right).$$
 It is in the multiplicative subgroup $1+tR[[t]]$ of the formal power series
 ring $R[[t]]$.
The trace formula for the twisted Lefschetz numbers implies
that $L^f_{\rho}(t)$ is a rational function in $R$ given by the formula
\begin{equation} \label{eq:determinant}
L^f_{\rho}(t)=\prod_{d}
          \det\big(E-t(z\tilde{F}_d)^{\rho})^{(-1)^{d+1}} \in R(t),
\end{equation} 
where $E$ stands for suitable identity matrices.
Twisted Lefschetz zeta function enjoys the same invariance properties
as that of $L_{H }(f^n)$.

\subsection{Computation of symplectic Floer homology }

In this section we describe known results from \cite{c1}, \cite{G}, \cite{f,ff} about computation 
of symplectic Floer homology for different mapping classes.

\subsubsection{Thurston classification theorem and standard form maps}
We recall firstly  Thurston classification theorem for homeomorphisms of surfase $M$
of genus $\geq 2$.

\begin{theorem}\label{thm:thur}\cite{Th}
Every homeomorphism $\phi: M\rightarrow M $ is isotopic to a homeomorphism $f$
such that either\\
(1) $f$ is a periodic map; or\\
(2) $f$ is a pseudo-Anosov map, i.e. there is a number $\lambda >1$, the dilation of $f$,  and a pair of transverse measured foliations $(F^s,\mu^s)$ and $(F^u,\mu^u)$ such that $f(F^s,\mu^s)=(F^s,\frac{1}{\lambda}\mu^s)$ and $f(F^u,\mu^u)=(F^u,\lambda\mu^u)$; or\\
(3)$f$ is reducible map, i.e. there is a system of disjoint simple closed curves\\ 
$\gamma=\{\gamma_1,......,\gamma_k\}$ in $int M$ such that $\gamma$ is invariant by $f$
(but $\gamma_i$ may be permuted) and $ \gamma$ has a $f$-invariant
tubular neighborhood $U$  such that each component of $M\setminus U$  has negative
Euler characteristic and on each(not necessarily connected) $f$-component of
$M\setminus U$, $f$ satisfies (1) or (2).

\end{theorem}
The map $f$ above is called a  Thurston canonical representative of $\phi$. In (3) it
can be chosen so that some iterate $f^m$ is a generalised Dehn twist on $U$.
A key observation is that if $f$ is canonical representative, so are all iterates of $f$.

Thurston classification theorem for homeomorphisms of surfase implies  that
every mapping class of $M$ is precisely one of the following:
 periodic, pseudo-Anosov or reducible.

In this section we review  standard form maps as discussed in \cite{G} and \cite{c1}.
These are special representative of mapping classes adopted to the symplectic geometry.
 For the identity mapping class, a standard form map is a small perturbation of the identity map by the Hamiltonian flow associated to a Morse function for which the boundary components are locally minima and maxima. Every fixed point is in the same Nielsen class. This Nielsen class has index given by the Euler characteristic of the surface. For non-identity periodic mapping classes, a standard form map is an isometry with respect to a hyperbolic structure on the surface with geodesic boundary. Every fixed point is in a separate Nielsen class and each of the Nielsen classes for which there is a fixed point has index one. For a pseudo-Anosov mapping classes, a standard form map is a symplectic smoothing(see \cite{c1}) of the singularities and boundary components of the canonical singular representative. Each singularity has a number $p\geq 3$ of prongs and each boundary component has a number $p \geq 1$ of prongs. If a singularity or boundary component is (setwise) fixed, it has some fractional rotation number modulo $p$(see \cite{c1}). There is a separate Nielsen class for every smooth fixed point, which is of index one or minus one; for every fixed singularity, which when symplectically smoothed gives $p-1$ fixed points all of index minus one if the rotation number is zero modulo $p$ or one fixed point of index one otherwise; and for every fixed boundary component with rotation number zero modulo $p$, which when symplectically smoothed gives $p$ fixed points all of index minus one \cite{c1}.

From this discussion, we see that for non-identity periodic and pseudo-Anosov mapping classes, the standard form map is such that all fixed points are nondegenerate of index +1 or -1 and, for every Nielsen class $F$, the number of fixed points in $F$ is $|\ind(F)|$. We now turn to reducible maps and the identity map.

By Thurston's classification (see \cite{Th} and \cite{flp}; also \cite[Definition 8]{G} and \cite[Definition 4.6]{c1}), in a reducible mapping class $g$, there is a (not necessarily smooth)  map $\phi$ which satisfies the following:

\begin{dfn}
\label{reduciblestandard}
A reducible map $\phi$ is in \emph{standard form} if there is a $\phi$-and-$\phi^{-1}$-invariant finite union of disjoint noncontractible (closed) annuli $U \subset M$ such that:
\be
\ii \label{twist} For $N$ a component of $U$and $\ell$ the smallest positive integer such that $\phi^\ell$ maps $N$ to itself, the map $\phi^\ell|_N$ is either a \emph{twist map} or a \emph{flip-twist map}. That is, with respect to coordinates $(q,p)\in[0,1]\times S^1$, we have one of \beqa(q,p)&\mapsto& (q,p-f(q)) \qquad \textrm{(twist map)} \\ (q,p)&\mapsto& (1-q,-p+f(q)) \qquad \textrm{(flip-twist map)},\eeqa
where $f : [0,1]\rightarrow \RR$ is a strictly monotonic smooth map. We call the (flip-)twist map \emph{positive} or \emph{negative} if $f$ is increasing or decreasing, respectively. Note that these maps are area-preserving.
\ii Let $N$ and $\ell$ be as in (\ref{twist}). If $\ell = 1$ and $\phi|_U$ is a twist map, then $\Im(f)\subset[0,1]$. That is, $\phi|_{\textrm{int}(N)}$ has no fixed points. (If we want to twist multiple times, we separate the twisting region into parallel annuli separated by regions on which the map is the identity.) We further require that parallel twisting regions twist in the same direction.
\ii For $S$ a component of $M \backslash N$ and $\ell$ the smallest integer such that $\phi^\ell$ maps $S$ to itself, the map $\phi^\ell|_S$ is area-preserving and is either isotopic to the \emph{identity},  \emph{periodic}, or \emph{pseudo-Anosov}. In these cases, we require the map to be in standard form as above.
\ee
\end{dfn}

Thurston classification theorem for homeomorphisms of surfase implies  that
every mapping class of $M$ is precisely one of the following:
 periodic, pseudo-Anosov or reducible.

\subsubsection{ Periodic mapping classes}

\begin{theorem}\label{thm:main}\cite{G}, \cite{ff}
If $\phi$ is  a non-trivial, orientation
preserving, standard form periodic diffeomorphism of a compact connected surface $M$
of Euler characteristic $\chi(M) \leq  0$, then $\phi$ is monotone symplectomorphism  with
respect to some $\phi$-invariant area form and
$$
\dim \HF(\phi) = L(\phi)= N (\phi)
$$
where  $L(\phi),  N (\phi)$ denote the Lefschetz and  the Nielsen  number of  $\phi$ correspondingly.
\end{theorem}

\subsubsection{Algebraically finite mapping classes}

 A mapping class of $M$ is called  algebraically finite if it
does not have any pseudo-Anosov components in the sense of Thurston's
theory of surface diffeomorphism.The term algebraically finite goes back to J. Nielsen \\
 In  \cite{G}  the diffeomorphisms  of finite type  were defined . These are reducible map
 in standard form which are
special representatives of algebraically finite mapping classes 
adopted to the symplectic geometry.

By  $M_{\id}$ we denote the union of the components
of $M\setminus\text{int}(U)$, where $\phi$ restricts to the identity.

The monotonicity of diffeomorphisms of finite type
was  investigated in details in \cite{G}.
Let $\phi$ be a diffeomorphism of finite type and
$\ell$ be as in (1).
Then $\phi^\ell$ is the product of (multiple)  Dehn twists along $U$.
Moreover, two parallel Dehn twists have the same sign. We say that
$\phi$ has  uniform twists, if $\phi^\ell$ is the product of only
positive, or only negative Dehn twists.
\smallskip\\
Furthermore, we denote by $\ell$ the smallest positive integer such that
$\phi^\ell$ restricts to the identity on $M\setminus U$.

If $\omega'$ is an area form on $M$ which is the standard form
$\de q\wedge\de p$ with respect to the $(q,p)$-coordinates on $U$, then
$\omega:=\sum_{i=1}^\ell(\phi^i)^*\omega'$ is
standard on $U$ and $\phi$-invariant, i.e. $\phi\in\symp(M,\omega)$.
To prove that $\omega$ can be chosen such that $\phi\in\symp^m(M,\omega)$,
Gautschi distinguishes two cases: uniform and non-uniform twists. In the first case he proves the following stronger statement.
\begin{lemma}\label{lemma:monotone3}\cite{G}
If $\phi$ has uniform twists and $\omega$ is a $\phi$-invariant area
form, then $\phi\in\symp^m(M,\omega)$.
\end{lemma}

In the non-uniform case, monotonicity  does not
hold for arbitrary $\phi$-invariant area forms.
\begin{lemma}\label{lemma:monotone4}\cite{G}
If $\phi$ does not have uniform twists, there exists a $\phi$-invariant
area form $\omega$ such that $\phi\in\symp^m(M,\omega)$. Moreover,
$\omega$ can be chosen such that it is the standard form $\de q\wedge\de p$ on
$U$.
\end{lemma}

\begin{theorem}\label{th:ga}\cite{G}
Let  $\phi$ be  a diffeomorphism of finite type, then $\phi$ is monotone with
respect to some $\phi$-invariant area form and
$$
\dim\HF(\phi) =
\dim H_*(M_{\id},\p_{M_{\id}};\Z_2) +
L(\phi|M\setminus M_{\id}).
$$
Here, $L$ denotes the Lefschetz number.
\end{theorem}

\subsubsection{ Pseudo-Anosov mapping classes}

For a pseudo-Anosov mapping classes, a standard form map is a symplectic smoothing of the singularities and boundary components of the canonical singular representative. Full description of the symplectic smoothing is given by A. Cotton-Clay in  \cite{c1}. Each singularity has a number $p\geq 3$ of prongs and each boundary component has a number $p \geq 1$ of prongs. If a singularity or boundary component is (setwise) fixed, it has some fractional rotation number modulo $p$. There is a separate Nielsen class for every smooth fixed point, which is of index one or minus one; for every fixed singularity, which when symplectically smoothed gives $p-1$ fixed points all of index minus one if the rotation number is zero modulo $p$ or one fixed point of index one otherwise; and there is a separate Nielsen class for every fixed boundary component with rotation number zero modulo $p$, which when symplectically smoothed gives $p$ fixed points all of index minus one.

\begin{theorem}\label{th:pa}( \cite{c1}, see also  \cite{ff})
If $\phi$ is any symplectomorphism with nondegenerate fixed points in
given pseudo-Anosov mapping class $ g $,  then
$\phi$ is weakly monotone, $\HF(\phi)$ is well defined  and
 $$ \dim\HF(\phi)=\dim\HF(g)=\sum_{x\in\Fix(\psi)}|\Ind(x)|,
$$
where $\psi$ is the singular canonical pseudo-Anosov representative of $g$. 
\end{theorem}

\subsubsection{Reducible mapping classes}

Recently, A. Cotton-Clay  \cite{c1} calculated
 Seidel's symplectic Floer homology for reducible mapping classes.
This result completing all previous computations.

In the  case of reducible mapping classes  a energy estimate forbids holomorphic discs from
crossing reducing curves except when a pseudo-Anosov component
meets an identity component ( with no twisting).
Let us introduce  some notation following \cite{c1}.
Recall the notation of $M_{id}$ for the collection of fixed components
as well as the tree types of boundary:  1) $\p_{+}M_{id}$, 
$\p_{-}M_{id}$ denote the collection of  components of $\p M_{id}$
on which we've joined up with a positive(resp. negative)  twist;
2) the collection of components of $\p M_{\id}$ which meet a pseudo-Anosov component will be denoted $\p_pM_{\id}$.
Additionally let $M_1$ be the collection of periodic components and let $M_2$ be the collection of pseudo-Anosov components with punctures( i.e. before any perturbation) instead of boundary components wherever there is a boundary component that meets a fixed component.
We further subdivide $M_{\id}$. Let $M_a$ be the collection of fixed components
which don't meet any pseudo-Anosov components. Let $M_{b,p}$ be the collection of fixed components which meet one pseudo-Anosov component at a boundary with
$p$ prongs. In this case, we assign the boundary components to
$\p_{+}M_{\id}$ (this is an arbitrary choice). Let $M^o_{b,p}$ be
the collection of the $M_{b,p}$ with each component punctured once.
Let $M_{c,q}$ be the collection of fixed components which meets at least
two pseudo-Anosov components such that the total number of prongs over all
the boundaries is $q$. In this case, we assign at least one boundary
component to $\p_{+}M_{\id}$ and at least one to $\p_{-}M_{\id}$
(and beyond that, it does not matter).

\begin{theorem} \label{th:red} \cite{c1}
If  $\bar{\phi}$ is a perturbed standard form map $\phi$ in a reducible
mapping class $g$ with choices of the signs of components of $\p_{p} M_{\id}$.Then  $\bar{\phi}$ is weakly monotone, $\HF(\bar{\phi})$ is well-defined and 
$$
\dim \HF(g)=\dim \HF(\bar{\phi})= \dim H_{*}(M_a,\p_{+}{M_{\id}};\Z_2) +
$$
$$
+ \sum_{p}(\dim H_{*}(M_{b,p}^0,\p_{+}M_{b,p};\Z_2) +(p-1)|\pi_{0}(M_{b,p})|) +
$$
$$
+ \sum_{q}(\dim H_{*}(M_{c,q},\p_{+}M_{c,q};\Z_2) +q|\pi_{0}(M_{c,q)})|) +
$$
$$
+L(\bar{\phi}|M_1)+ \dim HF_{*}(\bar{\phi}|M_2),
$$
where $L(\bar{\phi}|M_1)$ is the Lefschetz number of  $\bar{\phi}|M_1$,
the $L(\bar{\phi}|M_1)$ summand is all in even degree, the other
two  summands(with $p-1$ and $q$  are all in odd degree, and $HF_{*}(\bar{\phi}|M_2)$
denotes the Floer homology  for $\bar{\phi}$ on the pseudo-Anosov  components $M_2$

\end{theorem}

\begin{rk}
The first summand and the $L(\bar{\phi}|M_1)$ are as in R. Gautschi's 
Theorem \ref{th:ga} \cite{G}. The last summand comes from the  pseudo-Anosov components
and is calculated via  the  Theorem \ref{th:pa}. The sums  over $p$ and $q$ arise
in the same manner as the first summand.  
\end{rk}

\begin{cor}
As an application, A. Cotton-Clay gave recently \cite{c2} a sharp lower
bound on the number of fixed points of  area-preserving map  in any prescribed mapping class(rel boundary), generalising
the Poincare-Birkhoff fixed point theorem.
\end{cor}

\section{The growth rate of symplectic Floer homology }

\subsection{Topological entropy and Nielsen numbers}

The most widely used measure for the complexity of a dynamical system is the topological
entropy. For the convenience of the reader, we include its definition.
 Let $ f: X \rightarrow X $ be a self-map of a compact metric space. For given $\epsilon > 0 $
 and $ n \in \N $, a subset $E \subset X$ is said to be $(n,\epsilon)$-separated under $f$ if for
 each pair $x \not= y$ in $E$ there is $0 \leq i <n $ such that $ d(f^i(x), f^i(y)) > \epsilon$.
 Let $s_n(\epsilon,f)$  denote the largest cardinality of any $(n,\epsilon)$-separated subset $E$
 under $f$. Thus  $s_n(\epsilon,f)$ is the greatest number of orbit segments ${x,f(x),...,f^{n-1}(x)}$
 of length $n$ that can be distinguished one from another provided we can only distinguish 
 between points of $X$ that are  at least $\epsilon$ apart. Now let
 $$
 h(f,\epsilon):= \limsup_{n} \frac{1}{n}\cdot\log \,s_n(\epsilon,f)
 $$
 $$
 h(f):=\limsup_{\epsilon \rightarrow 0} h(f,\epsilon).
 $$
 The number $0\leq h(f) \leq \infty $, which to be independent of the metric $d$ used, is called the topological entropy of $f$.
 If $ h(f,\epsilon)> 0$ then, up to resolution $ \epsilon >0$, the number $s_n(\epsilon,f)$ of 
 distinguishable orbit segments of length $n$ grows exponentially with $n$. So $h(f)$
 measures the growth rate in $n$ of the number of orbit segments of length $n$
 with arbitrarily fine resolution.

   A basic relation between  topological entropy $h(f)$ and Nielsen numbers  was found by N. Ivanov
   \cite{I}. We present here a very short proof by Boju  Jiang  of the Ivanov's  inequality.
\begin{lemma}\label{lemma:ent}\cite{I}
$$
h(f) \geq \limsup_{n} \frac{1}{n}\cdot\log N(f^n)
 $$
\end{lemma}
\begin{proof}
 Let $\delta$ be such that every loop in $X$ of diameter $ < 2\delta $ is contractible.
 Let $\epsilon >0$ be a smaller number such that $d(f(x),f(y)) < \delta $ whenever $ d(x,y)<2\epsilon $. Let $E_n \subset X $ be a set consisting of one point from each essential fixed point class of $f^n$. Thus $ \mid E_n \mid =N(f^n) $. By the definition of $h(f)$, it suffices
 to show that $E_n$ is $(n,\epsilon)$-separated.
 Suppose it is not so. Then there would be two points $x\not=y \in E_n$ such that $ d(f^i(x), f^i(y)) \leq \epsilon$ for $o\leq i< n$ hence for all $i\geq 0$. Pick a path $c_i$ from $f^i(x)$ to
 $f^i(y)$ of diameter $< 2\epsilon$ for $ o\leq i< n$ and let $c_n=c_0$. By the choice of $\delta$
 and $\epsilon$ ,  $f\circ c_i \simeq c_{i+1} $ for all $i$, so $f^n\circ c_0\simeq c_n=c_0$. 
 This means $x,y$ in the same fixed point class of $f^n$, contradicting the construction of $E_n$.

\end{proof}

 This inequality is remarkable in that it does not require smoothness of the map and provides a common lower bound for the topological entropy of all maps in a homotopy class.

\subsection{Asymptotic invariant}

Let $\Gamma = \pi_0(Diff^+(M))$ be the mapping class group of a closed connected oriented surface $M$ of genus $\geq 2$. Pick an everywhere positive two-form $\omega$ on $M$. A isotopy  theorem of Moser \cite{Mo} says that each  mapping class of  $g \in \Gamma$,  i.e. an isotopy class of $Diff^+(M)$,  admits representatives which preserve $\omega$. Due  to Seidel\cite{S} and Cotton-Clay \cite{c1} we can
pick  a monotone(weakly monotone)  representative $\phi\in\symp^m(M,\omega)$( or $\phi\in \symp^{wm}(M,\omega)$) of $g$ such that $\HF(\phi)$ is an invariant as $\phi$ is deformed through monotone(weakly monotone) symplectomorphisms.
These imply that we have a symplectic Floer homology invariant $\HF(g)$ canonically assigned to each
mapping class  $g$ given by $\HF(\phi)$ for any monotone(weakly monotone) symplectomorphism $\phi$.

Note that $\HF(g)$ is independent of the choice of an area
form $\omega$ by Moser's  theorem  and naturality of Floer homology. 
 
Taking a dynamical point of view,
 we consider now  the iterates of monotone(weakly monotone)  symplectomorphism  $\phi$.
Symplectomorphisms  $\phi^n$
are  also monotone(weakly monotone) for all $n>0$ \cite{G, c1}.

The growth rate of a sequence $a_n$ of complex numbers is defined
by
 $$
 \grow ( a_n):= max \{1,  \limsup_{n \rightarrow \infty} |a_n|^{1/n}\}
$$

which could be infinity. Note that $\grow(a_n) \geq 1$ even if all $a_n =0$.
When  $\grow(a_n) > 1$, we say that the sequence $a_n$ grows exponentially.

In \cite{f}  we  have introduced  the asymptotic invariant $ F^{\infty}(g)$ assigned  to  mapping class   $g \in Mod_M = \pi_0(Diff^+(M))$  via  the growth rate of the sequence
$\{a_n=\dim\HF(\phi^n)\}$ for  a monotone( or weakly monotone)  representative $\phi\in\symp^m(M,\omega)$ of $g$:
$$ F^{\infty}(g):=\grow(\dim\HF(\phi^n)) $$

\begin{ex} \label{ex:1}
If $\phi$ is  a non-trivial orientation
preserving standard form periodic diffeomorphism of a compact connected surface $M$
of Euler characteristic $\chi(M) <  0$ , then the  periodicity of the
sequence $\dim\HF(\phi^n)$ implies that  for the  corresponding  mapping class $g$ the  asymptotic invariant
$$ F^{\infty}(g):=\grow(\dim\HF(\phi^n))=1 $$

\end{ex}

\begin{ex}\label{ex:2} Let $\phi$ be  a monotone  diffeomorphism of finite type of a compact connected surface $M$
of Euler characteristic $\chi(M) <  0$ and $g$ a corresponding algebraically finite mapping class.
Then the total dimension of $\HF(\phi^n)$  grows at most linearly (see \cite{c1, Sm, ff}. 
Taking the growth rate in $n$, we get
that  the asymptotic invariant  $$ F^{\infty}(g):=\grow(\dim\HF(\phi^n))=1$$.
\end{ex}

 For any set $S$ let $\Z S$ denote the free Abelian group with the specified basis $S$.The norm in $\Z S$ is defined by
\begin{equation}
\|\sum_i k_is_i\|:= \sum_i \mid k_i\mid \in \Z,
\end{equation}
when the $s_i$ in $S$ are all different.
 
For a $\Z H $-matrix $ A=(a_{ij}) $,define its norm by $ \| A \|:=\sum_{i,j}\| a_{ij} \| $.Then we have inequalities $\| AB \|\leq \| A\ |\cdot\| B \|$ when $A,B$ can be multiplied, and $\| \tr A \|\leq \| A \| $ when $A$ is a square matrix.For a matrix  $ A=(a_{ij}) $ in $\Z S$, its matrix of norms is defined to be the matrix $ A^{norm}:=(\| a_{ij} \|)$ which is a matrix of non-negative integers.In what follows, the set $S$ wiill be $\pi $, $H $ or
$H _c$.We denote by $s(A)$ the spectral radius of $A$, $s(A)=\lim_n \sqrt[n]{\| A^n \| |}$ which coincide with the largest  modul of an eigenvalue of $A$.
 
\begin{rk} The norm $\| L_{H}(f^n) \| $ is the sum of absolute values of the indices of all the $n$-orbits classes $O^n$ . It equals $\| L_{\pi}(f^n) \|$, the sum of absolute values of the indices of all the fixed point classes of $f^n$, because any two fixed point classes of $f^n$ contained in the same $n$-orbit class $O^n$ must have the same index. The norm $\| L_{\pi}(f^n) \|$ is  homotopy type invariant.   

We define the asymptotic absolute Lefschetz number \cite{j} to be the growth rate
 $$
L^{\infty}(f)=  \grow(\| L_{\pi}(f^n) \|)
$$
We also define the asymptotic Nielsen  number\cite{I} to be the growth rate
 $$
N^{\infty}(f)=  \grow(N(f^n))
$$
All these asymptotic numbers are homotopy type invariants.

\end{rk}

\begin{lemma}\label{lem:pa}
If $\phi$ is  any symplectomorphism with nondegenerate fixed points in
given pseudo-Anosov mapping class $g $,  then
$$ \dim\HF(\phi)=\dim\HF(g)=\|L_{\pi}(\psi) \|,
$$
where $\psi$ is a singular canonical  pseudo-Anosov representative of $g$. 

\end{lemma}

{\it Proof.} It is known that for pseudo-Anosov map $\psi$ fixed points are topologically separated,
 i.e.  each essential  fixed point class of $\psi$ consists of a single fixed point(see  \cite{Th,I,FelshB}).
 Then  the generalized Lefschetz number  or the Reidemeister trace \cite{j} is 
\begin{equation} 
L_{\pi}(\psi):=\sum_{F}\ind (F,\psi)\cdot cd_{\pi}(F,\psi) =\sum_{x\in\Fix(\psi)}\ind (x)\cdot cd_{\pi}(x,\psi)\in  \Z\pi_\psi , 
\end{equation} 
where the summation being over all essential fixed point classes $F$ of $\psi$ i.e over all fixed points of $\psi$.
So,  the result follows from the theorem \ref{th:pa} and  the definition of the norm $\|L_{\pi}(\psi) \|$.

\begin{rk}
Lemma \ref{lem:pa} provides via Reidemeister trace formula  a  new combinatorial formula to compute  $\dim\HF(g)$ comparable to the  train-track combinatorial formula
of A. Cotton-Clay in \cite{c1}.

\begin{theorem}\label{th:I}\cite{flp,I,j} Let $f$ be a pseudo-Anosov  homeomorphism with dilatation  $\lambda >1$ of surfase $M$
of genus $\geq 2$. Then
$$h(f)=log(\lambda)=\log N^{\infty}(f)=\log L^{\infty}(f) $$
\end{theorem}

\begin{theorem}\label{th:j}\cite{j}
Suppose $f$ is canonical representative of a homeomorphism  of surfase $M$
of genus $\geq 2$ and  $\lambda$ is the largest  dilatation  of the  pseudo-Anosov components( $\lambda=1$ if there is no pseudo-Anosov components).
Then
$$h(f)=log(\lambda)=\log N^{\infty}(f)=\log L^{\infty}(f)$$
\end{theorem}

\end{rk}
\begin{theorem}\label{th:pa1}  If $\phi$ is  any  symplectomorphism with nondegenerate fixed points in
given pseudo-Anosov mapping class $ g$  with dilatation  $\lambda >1$ of surfase $M$
of genus $\geq 2$. Then

 $$ F^{\infty}(g):=\grow(\dim\HF(\phi^n))=\lambda=\exp(h(\psi))= L^{\infty}(\psi)=N^{\infty}(\psi) $$
where $\psi$ is a canonical singular  pseudo-Anosov representative of $g$.
\end{theorem}

{\it Proof.} By  lemma \ref{lem:pa} we have that  $\dim\HF(\phi^n)=\|L_{\pi}(\psi^n) \|$ for every $n$.
So,  the result follows from theorem \ref{th:I}.

\begin{theorem}\label{th:main} Let $\bar{\phi}$ be a perturbed standard form map $\phi$ in a reducible
mapping class $g$ of compact surface  of genus $\geq 2$ and  $\lambda$ is the largest  dilatation  of the  pseudo-Anosov components( $\lambda=1$ if there is no pseudo-Anosov components). Then

 $ F^{\infty}(g):=\grow(\dim\HF((\bar{\phi})^n))=\lambda=\exp(h(\psi))= L^{\infty}(\psi)=N^{\infty}(\psi)$
 where $\psi$ is canonical representative of mapping class $g$.
\end{theorem}

{\it Proof.} It follows  from the theorem \ref{th:red}  that for every n 
$$
\dim \HF(g^n)=\dim \HF((\bar{\phi})^n)= \dim H_{*}(M_a,\p_{+}{M_{\id}};\Z_2) +
$$
$$
+ \sum_{p}(\dim H_{*}(M_{b,p}^0,\p_{+}M_{b,p};\Z_2) +(p-1)|\pi_{0}(M_{b,p})|) +
$$
$$
+ \sum_{q}(\dim H_{*}(M_{c,q},\p_{+}M_{c,q};\Z_2) +q|\pi_{0}(M_{c,q)})|) +
$$
$$
+L((\bar{\phi})^n|M_1)+ \dim HF_{*}((\bar{\phi})^n|M_2).
$$

We need   to investigate only  the growth  of the last summand in this formula  because the rest part  in the formula  grows at most linearly \cite{ c1, Sm}.
We have $\dim HF_{*}((\bar{\phi})^n|M_2)=\sum_{j}\dim HF_{*}((\bar{\phi_j})^n|M_2)$, where the sum is taken over different  pseudo-Anosov components of 
$\phi^n|M_2 $.
It follows from the  theorem \ref{th:pa1} that  $\dim HF_{*}((\bar{\phi_j})^n|M_2)$ grows as $\lambda_j^n$, where $\lambda_j$ is the dilatation of the pseudo-Anosov component $\bar{\phi_j}|M_2$.

Taking growth rate  in $n$, we get $ F^{\infty}(g):=\grow(\dim\HF((\bar{\phi})^n))=\max_{j}\lambda_j=\lambda=\exp(h(\psi))$.

\begin{cor}\label{cor:entropy} The asymptotic invariant  $F^{\infty}(g) > 1 $ if and only if $\psi$ has a pseudo-Anosov component.

\end{cor}

Although the exact evaluation of the asymptotic invariant $F^{\infty}(g)$ would be desirable, its estimation is a more realistic goal. We carry out such estimation using notations and results from
sections \ref{section: Reidemeister} and \ref{section: Lefschetz}.

\begin{prop} \label{prop:main} 
Suppose  $\rho:  H \rightarrow U(n)$ is a unitary representation and $\psi$ is canonical representative of
reducible mapping class $g$ of compact surface  of genus $\geq 2$. Let $w$ be a zero or a pole
of the rational function  $L^{\psi}_{\rho}(t)\in \C (t)$     . Then
$$
 \frac{1}{\mid w \mid}  \leq   F^{\infty}(g)  \leq  \max_d \| z\tilde F_d \|
$$

\end{prop}
{\it Proof.} We know from complex analysis and definition of the twisted 
Lefschetz zeta function that $Growth (L_{\rho}(\psi^n))$ is the reciprocal of the radius
of convergence of the function $log(L^{\psi}_{\rho}(t)$, hence $Growth( L_{\rho}(\psi^n))\geq \frac{1}{\mid w \mid} $. 

On other hand, according to section \ref{section: Reidemeister}, $H$-coordinates of $n$-orbit
classes are in the form $[z^n]g$ with $g\in\pi$. So we can assume 
$L_H(\psi^n)=\sum_i k_i[z^ng_i]$, where the $[z^ng_i]$ are different conjugacy classes in $H$.
Since the trace of a unitary matrix is bounded by its dimension, we get that 
$L_{\rho}(\psi^n) \in \C $ are bounded by $$\mid L_{\rho}(\psi^n)\mid =\mid \sum_i k_i\tr(z^ng_i)^\rho\mid
\leq \sum_i \mid k_i\mid\mid\tr(z^ng_i)^\rho\mid \leq \sum_i \mid k_i\mid=\|L_H(\psi^n)\|.$$
Hence $Growth (L_{\rho}(\psi^n))\leq L^{\infty}(\psi).$ 
From theorem \ref{th:main}  it follows that $ F^{\infty}(g) =L^{\infty}(\psi)$.
So we get the estimation from below $ \frac{1}{\mid w \mid}  \leq   F^{\infty}(g). $
The initial data of our lower estimation is the knowledge of the $\Z H$-matrices $\tilde F_d $
provided by a cellular map, which enables us to compute the twisted Lefschetz zeta function.
There is also a way to derive an upper bound from the same data. We have
$$
\dim\HF(g^n)= \|L_{\pi}(\psi^n)\|\!=\!\|L_H(\psi^n)\|\!=\!\|\sum_d(-1)^d[\tr(z\tilde F_d)^n]\|\!\leq\sum_d\| [\tr(z\tilde F_d)^n]\|\!
$$
$$
\leq\sum_d{||}\tr(z\tilde F_d)^n{||}\leq \sum_d\tr((z\tilde
F_d)^n)^{norm}\leq\sum_d\tr((z\tilde F_d)^{norm})^n
$$
$$
\leq \sum_d\tr((\tilde F_d)^{norm})^n.
$$
Hence $$F^{\infty}(g) =Growth(\|L_{\pi}(\psi^n)\|)=Growth(\|L_{H}(\psi^n)\|)\leq Growth(\sum_d \tr((\tilde F_d)^{norm})^n)
$$
$$= max_d ( Growth(\tr((\tilde F_d)^{norm})^n)) =max_d  (s(\tilde F_d)^{norm}).$$

\begin{rk}
A practical difficulty in the use of $L_{\rho}(\psi^n))$ and twisted Lefschetz zeta
function $L^{\psi}_{\rho}(t)$ for  estimation is to find a useful representation $\rho$.
Following  approach of Boju Jiang in \cite{j}, section 1.7  we can weaken the assumption on $\rho$ in proposition \ref{prop:main}. There are many examples in \cite{j}, chapter 4 which illustrate method of estimation above
for surface homeomorphisms.
\end{rk}

\subsection{Symplectic Floer homology and geometry of 3-manifolds}

A three dimensional manifold $M^3$ is called a graph manifold if there is a system of mutually disjoint two-dimensional tori $ T_i$ in $M^3$ such that the closure of each component of $M^3$
 cut along  union of tori $T_i$ is a product of surface  and $S^1$.

\begin{theorem}\label{teo:graph}
 Let $\chi(M) < 0 $. The mapping torus $T_{\phi}$ is a graph-manifold if and only if 
asymptotic invariant $F^{\infty}(g)=1$.
 If $ Int (T_{\phi}) $ admits a hyperbolic structure of finite volume, then asymptotic invariant $F^{\infty}(g) > 1$. If asymptotic invariant  $F^{\infty}(g )> 1$ then 
 $\phi$ has an infinite set of periodic points with pairwise different periods.
\end{theorem}
{\it Proof.}
T. Kobayashi \cite {koba} has proved that mapping torus $T_{\phi}$ is a graph-manifold if and only if the
Thurston canonical representative
for $\phi$ does not contain pseudo-Anosov components. So, Example \ref{ex:2}  and  Corrolary \ref{cor:entropy}  implie the first statement of the theorem. Thurston has proved \cite {thu3}, \cite{su} that  $ Int (T_{\phi})$  admits a hyperbolic structure of finite volume if and only if $\phi$ is isotopic to pseudo-Anosov homeomorphism.
This proves the second statement of the theorem.  It is known \cite{koba} that
pseudo-Anosov homeomorphism has infinitely many periodic points those periods are
mutually distinct. This proves the last  statement of the theorem.

Recall that the set of isotopy classes of orientation preserving homeomorphisms $\phi:M
\to M$ forms a group called the mapping class group, denoted $\Mod(M)$.  This group acts properly
discontinuously by isometries on the Teichm\"uller space $\Teich(M)$ with quotient the moduli space $\M(M)$ of Riemann
surfaces homeomorphic to $M$.  The closed geodesics in the orbifold $\M(M)$ correspond precisely to the conjugacy
classes of mapping classes represented by pseudo-Anosov homeomorphisms, and moreover, the length of a geodesic
associated to a pseudo-Anosov $\phi: M\to M$ is $\log(\lambda(\phi))$. 
We define  the Floer spectrum of $\M(M)$ as the set 
\[ \spec_F(\Mod(M))=\{\log(F^{\infty}(g )): g \mbox{ is pseudo-Anosov mapping class}\} \subset (0,\infty).\]
 
 By theorem \ref{th:pa1} 
 the Floer spectrum of $\M(M)$ coincides with the length spectrum 

\[ \spec(\Mod(M))=\{\log(\lambda(\phi)): \phi:M \to M\mbox{ is pseudo-Anosov} \} \subset (0,\infty).\]

Arnoux--Yoccoz \cite{AY} and Ivanov \cite{Iv} proved that $\spec(\Mod(M))=\spec_F(\Mod(M))$ is a closed discrete subset of $\R$. It
follows that $\spec_F(Mod(M))$ has, for each $M$, a least element, which we shall denote by $F(M)$.  We can think of
$F(M)$ as the {\em systole} of $\M(M)$.

From result of Penner\cite{Pe} it is follows  that there exists constants $0 < c_0 < c_1$ so that for all closed surfaces $M$ with $\chi(M) < 0$, one has
\begin{equation}
\label{eq:bounding dil} c_0 \leq F(M)   |\chi(M)|\leq c_1.
\end{equation}
The proof of the lower bound comes from a spectral estimate for Perron--Frobenius matrices, with $c_0 > \log(2)/6$ (see \cite{Pe} and \cite{Mc2}).   As such, this lower bound is valid for all surfaces $M$ with $\chi(M) < 0$,
including punctured surfaces.  The upper bound is proven by constructing pseudo-Anosov homeomorphisms $\phi_g:M_g \to
M_g$ on each closed surface of genus $g \geq 2$ so that $\lambda(\phi_g) \leq e^{c_1/(2g-2)}$.

The best known upper bound for $\{F(M_g) |\chi(M_g)|\}$ follows from Hironaka - Kin \cite{HK} and from Minakawa
\cite{Mk}, and is $2 \log(2 + \sqrt{3})$.  The situation for punctured surfaces is more mysterious.

For a pseudo-Anosov  $\phi$, let $\tau_{WP}(\phi)$ denote the 
translation length of $\phi$, thought of as an
isometry of $\Teich(M)$ with the Weil--Petersson metric. Brock \cite{Br} 
has proven that the volume of the mapping torus $T_\phi$ and
$\tau_{WP}(\phi)$ satisfy a bilipschitz relation, and in particular
\[\vol(M_\phi) \leq c \tau_{WP}(\phi).\]
Moreover, there is a relation between the Weil--Petersson
translation length and the Teichm\"uller translation length  $\tau_{\Teich}(\phi) = \log(\lambda(\phi))$ (see
\cite{Li}), which implies
\[ \tau_{WP}(\phi) \leq \sqrt{2\pi |\chi(S)|} \log(\lambda(\phi)). \]

Then the theorem \ref{th:pa1} implies  the following estimation
\[\vol(T_\phi) \leq c \tau_{WP}(\phi)\leq c\sqrt{2\pi |\chi(M)|} \log(F^{\infty}(g )),\]  where $g$ is pseudo-Anosov mapping class of $\phi$.

However, Brock's constant $c=c(M)$ depends on the surface $M$, and moreover $c(M) \geq |\chi(M)|$ when $|\chi(M)|$ is sufficiently large.

\subsection{Radius of convergence of the symplectic  zeta function}
 
In \cite{f} we have introduced  a symplectic zeta function 

\begin{eqnarray*}
F_g(t)= F_\phi(t) &= &
 \exp\left(\sum_{n=1}^\infty \frac{\dim\HF(\phi^n)}{n} t^n \right)
\end{eqnarray*}

assigned to mapping class $g$ via  zeta function $F_\phi(t)$  of  a monotone( or weakly monotone)  representative $\phi\in\symp^m(M,\omega)$ of $g$. Symplectomorphisms  $\phi^n$
are  also monotone(weakly monotone) for all $n>0$ \cite{G, c1} so,  symplectic zeta function
 $ F_\phi(t)$ is an invariant as $\phi$ is deformed through monotone(or weakly monotone) symplectomorphisms in $g$.
These imply that we have a symplectic Floer homology invariant  $ F_g(t)$  canonically assigned to each
mapping class  $g$.  
A motivation for the definition of this zeta function was  a connection \cite{f,G} between Nielsen numbers  and Floer homology and nice
analytic properties of Nielsen zeta function \cite{pf,fv,f,fl,FelshB,fh1,fh2} 

We denote by $R$ the radius of convergence of the symplectic  zeta function $F_g(t)= F_\phi(t)$.

In this section we  give  exact  algebraic lower estimation for the radius $R$
using Reidemeister trace formula   for generalized Lefschetz numbers from section \ref{section: Reidemeister}

\begin{theorem}
If $\phi$ is any symplectomorphism with nondegenerate fixed points in
given pseudo-Anosov mapping class $ g$  with dilatation  $\lambda >1$  of  compact surface   $M$  of genus $\geq 2$, then  the symplectic  zeta function $F_g(t)$ has positive radius of convergence $R=\frac{1}{\lambda}$.
Radius of convergence $R$ admits following estimations 
\begin{equation}
 R\geq \frac{1}{\max_d \| z\tilde F_d \|} > 0
\end{equation}
and
\begin{equation}
R\geq \frac{1}{\max_d s(\tilde F_d^{norm})} > 0 
\end{equation}
 
\end{theorem}
 
{\it Proof.}
 It follows from lemma \ref{lem:pa} that   $\dim\HF(\phi^n)=\|L_{\pi}(\psi^n) \|$.
By the homotopy type invariance of the right hand side we can estimate it.  We can suppose that $\psi$ is a cell map of a finite cell complex.The norm $\| L_{H}(\psi^n) \| $ is the sum of absolute values of the indices of all the $n$-orbits classes $O^n$ . It equals $\| L_{\pi}(\psi^n) \|$, the sum of absolute values of the indices of all the fixed point classes of $\psi^n$, because any two fixed point classes of $\psi^n$ contained in the same $n$-orbit class $O^n$ must have the same index. From this we have $\dim\HF(\phi^n)=\|L_{\pi}(\psi^n)\|=\|L_H(\psi^n)\|=\|\sum_d(-1)^d[\tr(z\tilde F_d)^n]\|\leq\sum_d\|[\tr(z\tilde F_d)^n]\|\leq\sum_d\|\tr(z\tilde F_d)^n\|\leq\sum_d\|(z\tilde F_d)^n\|\leq\sum_d\|(z\tilde F_d)\|^n $. The radius of convergence $ R$ is given by Caushy-Adamar formula:
$$
 \frac{1}{R}=  \limsup_n \sqrt[n]{\frac{\dim\HF(\phi^n)}{n}}=\limsup_n \sqrt[n]{\dim\HF(\phi^n)}=\lambda.
$$
Therefore we have:
$$
R= \frac{1}{\limsup_n \sqrt[n]{\dim\HF(\phi^n)}}\geq \frac{1}{\max_d \| z\tilde F_d \|} > 0.
$$
Inequalities:
$$
\dim\HF(\phi^n)= \|L_{\pi}(\psi^n)\|\!=\!\|L_H(\psi^n)\|\!=\!\|\sum_d(-1)^d[\tr(z\tilde F_d)^n]\|\!\leq\sum_d\| [\tr(z\tilde F_d)^n]\|\!
$$
$$
\leq\sum_d{||}\tr(z\tilde F_d)^n{||}\leq \sum_d\tr((z\tilde
F_d)^n)^{norm}\leq\sum_d\tr((z\tilde F_d)^{norm})^n
$$
$$
\leq \sum_d\tr((\tilde F_d)^{norm})^n
$$
and the definition of spectral radius give estimation:
$$
R= \frac{1}{\limsup_n \sqrt[n]{\dim\HF(\phi^n)}} \geq \frac{1}{\max_d s(\tilde F_d^{norm})} > 0.
$$
  \vskip 4pt plus 2pt

\label{th:main2}
\begin{theorem}
 Let $\bar{\phi}$ be a perturbed standard form map $\phi$ in a reducible
mapping class $g$ of compact surface  of genus $\geq 2$ and  $\lambda$ is the largest  dilatation  of the  pseudo-Anosov components( $\lambda=1$ if there is no pseudo-Anosov components). Then
 the symplectic  zeta function $F_g(t)= F_{\bar{\phi}}(t)$ has positive radius of convergence  $R=\frac{1}{\lambda}$, where  $\lambda$ is the largest  dilatation  of the  pseudo-Anosov components( $\lambda=1$ if there is no pseudo-Anosov components).
\end{theorem}
{\it Proof.}
 The radius of convergence $ R$ is given by Caushy-Adamar formula:
$$
 \frac{1}{R}=  \limsup_n \sqrt[n]{\frac{\dim\HF(\phi^n)}{n}}=\limsup_n \sqrt[n]{\dim\HF(\phi^n)}.
$$ By Theorem \ref{th:main} we have 
$$
\limsup_n \sqrt[n]{\dim\HF(\phi^n)}=\grow(\dim\HF(\phi^n))=\lambda
$$

\begin{ex}
 Let $X$ be surface with boundary, and  $f:X\rightarrow X$ be a map.Fadell and Husseini(see  \cite{j}) devised a method of computing the matrices of the lifted chain map for surface maps.Suppose $ \{a_1, .... ,a_r\} $ is a free basis for $\pi_1(X)$. Then $X$ has the homotopy type of a bouquet $B$  of $r$ circles which can be decomposed into one 0-cell and $r$  1-cells corresponding to the $a_i$,and $f$ has the homotopy type of a cellular map  $g:B\rightarrow B.$
By the homotopy type invariance of the invariants,we can replace $f$ with $g$ in computations.The homomorphism $\tilde f_*:\pi_1(X)\rightarrow \pi_1(X)$ induced by $f$ and $g$ is determined by the images $b_i=\tilde f_*(a_i), i=1,.. ,r $.The fundamental group $\pi_1(T_f)$ has a presentation $\pi_1(T_f)=<a_1,...,a_r,z| a_iz=zb_i, i=1,..,r>$.Let
$$
D=(\frac{\partial b_i}{\partial a_j})
$$
be the Jacobian in Fox calculus(see \cite{j}).Then,as pointed out in \cite{j}, the matrices of the lifted chain map $\tilde g$ are
$$
\tilde F_0=(1), \tilde F_1=D=(\frac{\partial b_i}{\partial a_j}).
$$
Now, we can find estimations for the radius $R$ as above.
 \end{ex}

 Let $ \mu(d), d \in \N$,
be the M\"obius function.

 \begin{theorem}\cite{f}
Let $\phi$ be  a non-trivial orientation
preserving standard form periodic diffeomorphism  of least period $m$ of a compact connected surface $M$
of Euler characteristic $\chi(M) < 0$  . Then the symplectic
  zeta function $ F_g(t)= F_\phi(t)$ is  a radical of a rational function and 
$$
F_g(t)= F_\phi(t) =\prod_{d\mid m}\sqrt[d]{(1-t^d)^{-P(d)}},
$$
 where the product is taken over all divisors $d$ of the period $m$, and $P(d)$ is the integer
$  P(d) = \sum_{d_1\mid d} \mu(d_1)\dim\HF(\phi^{d/ d_1}) .  $
\end{theorem}

We denote by $L_\phi(t)$ the  Weil  zeta function 
$$
 L_\phi(t)  :=  \exp\left(\sum_{n=1}^\infty \frac{L(\phi^n)}{n} t^n \right),
$$ where $L(\phi^n)$ is the Lefschetz  number of  of $\phi^n$.

\begin{theorem}\cite{f}
If  $\phi$ is a  hyperbolic  diffeomorphism of a 2-dimensional torus $T^2$,
then the  symplectic zeta function $ F_g(t)=F_\phi(t)$ is  a rational function
and $F_g(t)=F_\phi(t)=(L_\phi(\sigma\cdot t))^{(-1)^r}$ , where  $r$ is equal to the number of $\lambda_i \in Spec(\tilde \phi) $ such that $ \mid \lambda_i \mid > 1$, $p$ is equal to the number of $\mu_i \in Spec(\tilde \phi)$ such that $\mu_i <-1$ and  $\sigma=(-1)^p$, here $\tilde \phi$ is a lifting of $\phi$
to the universal cover.
\end{theorem}

 \subsection{ Concluding remarks and questions}

\begin{rk}
For a symplectic manifold $X$  the (conjugation-invariant) \emph{Floer-type entropy} of $g\in Symp(X)/Ham(X)$, a mapping class  of  $\phi $, is defined in \cite{Sm} as 
$$h_{F}(g) \ = \ \limsup \frac{1}{n} \log rk\, HF(\phi^n)= \log F^{\infty}(g)
$$
This is a kind of robust version of the periodic entropy, robust in the sense that it depends on a symplectic diffeomorphism only through its mapping class; by contrast topological and periodic entropy are typically very sensitive to perturbation.  
As we proved above, for area-preserving diffeomorphisms of a surface $M$, the Floer-type entropy  coincides with the topological  entropy of the canonical representative  in corresponding mapping class; moreover, $h_{F}(g)>0 $ if and only if  $ g$ has a pseudo-Anosov component.  
\end{rk}

\begin{que}(Entropy conjecture for symplectomorphisms)

Is it always true that for symplectomorphisms of compact symplectic manifolds
$$h(\phi)\geq \log F^{\infty}(g)=\log\grow(\dim\HF(\phi^n)) = h_F(g) ?$$ 
\end{que} 
\begin{que}(A weak version of the Entropy conjecture for symplectomorphisms)

Is it always true that for symplectomorphisms of compact symplectic manifolds
$$h(\phi)\geq \log\grow(\mid\chi(\HF(\phi^n))\mid)  ?$$ 
\end{que}  
Here $\chi(\HF(\phi^n))$ is the Euler characteristic of symplectic Floer homology of $\phi^n$. 
 If for every $n$ all the fixed points of $\phi^n$ are non-degenerate, i.e. for all $x\in\fix(\phi^n)$, $\det(\id-\de\phi^n(x))\ne0$, then  
$$\chi(\HF(\phi^n))=\sum_{x=\phi^n(x)} \sign(\det(\id-\de\phi^n(x)))=L(\phi^n).
$$ This implies that the  question above is a version of the question of Shub \cite{sh}.

\begin{que} 
Is it true that  for a symplectomorphism $\phi$ of an aspherical compact symplectic manifold
 $$ F^{\infty}(g):=\grow(\dim\HF(\phi^n))= L^{\infty}(\phi)=N^{\infty}(\phi) ?$$ 
\end{que}

Inspired by the Hasse-Weil  zeta function of an algebraic variety over a finite field,
  Artin and Mazur \cite{am} defined  the  zeta function for an arbitrary map $f: X \rightarrow X $
  of a topological space $X$:
 $$ 
 AM_f(t)  :=  \exp\left(\sum_{n=1}^\infty \frac{\#\Fix(f^n)}{n} t^n \right),
 $$
  where $\#\Fix(f^n)$ is the number of isolated fixed points of $f^n$.
Artin and Mazur showed that for a dense set of the space of smooth maps of a compact smooth manifold into itself the  number of periodic points  $\#\Fix(f^n)$ grows at most exponentially and  the  Artin-Mazur zeta function $ AM_f(t) $ has a positive radius of convergence \cite{am}. Later   Manning \cite{m} proved the rationality of the Artin - Mazur zeta function for diffeomorphisms of a smooth compact manifold satisfying Smale  axiom A. On the other hand there exist maps for which Artin-Mazur zeta function is transcendental .
 The symplectic zeta function  $F_\phi(t)$  can be considered as some analog of  the Artin-Mazur zeta function  $ AM_f(t) $  because periodic points of $\phi^n$ provide  the generators of symplectic Floer homologies    
 $\HF(\phi^n)$. This motivate following 

\begin{conj} For any compact symplectic manifold $M$ and  symplectomorphism $\phi: M \rightarrow M$   with  well defined Floer homology groups $\HF(\phi^n)$,  $n\in \N$ the symplectic
zeta function $F_g(t)=F_\phi(t)$ has a  positive radius of convergence.

\end{conj}

\begin{que}

Is the symplectic zeta function  $F_g(t)=F_\phi(t)$ an algebraic function  of $z$?

\end{que}

\begin{rk}
Given a symplectomorphism $\phi$ of surface $M$, one can form
the symplectic mapping torus
$M^4_{\phi}=T^3_{\phi}\rtimes S^1$, where $T^3_{\phi}$ is  usual mapping torus
.
Ionel and Parker \cite{IP} have computed the degree zero Gromov invariants
\cite{IP}(these are built from the invariants of Ruan and Tian)  of
$M^4_{\phi}$ and of fiber sums of the $M^4_{\phi}$ with other symplectic manifolds. This is done by expressing the Gromov invariants in terms of the
Lefschetz zeta function $ L_\phi(z)$ \cite{IP}. The result is a large set of interesting non-Kahler
symplectic manifolds with computational ways of distinguishing them. In dimension four this gives a symplectic construction of the exotic elliptic surfaces of Fintushel and Stern \cite{FS}. This construction arises from knots.
Associated to each fibered knot $K$ in $S^3$ is a Riemann surface $M$ and a
monodromy diffeomorphism $f_K$ of $M$. Taking $\phi=f_K$ gives symplectic
4-manifolds ${M^4}_\phi(K)$ with Gromov invariant
$Gr({M^4}_\phi(K))= A_K(t)/(1-t)^2=L_\phi(t)$, where $ A_K(t)$ is the Alexander polynomial of knot $K$. Next, let $E^4(n)$ be the simply-connected minimal elliptic surface with fiber $F$ and canonicla divisor $k=(n-2)F$. Forming the fiber sum $E^4(n,K)=E^4(n)\#_{(F=T^2)}{M^4}_\phi(K)$ we obtain a
symplectic manifold homeomorphic to $E^4(n)$.
Then for $n\geq 2$ the Gromov and Seiberg-Witten invariants of $E^4(K)$
are $ Gr(E^4(n,K))=SW(E^4(n,K))=A_K(t)(1-t)^{n-2}$ \cite{FS, IP}. 
Thus fibered knots with distinct Alexander polynomials give rise to symplectic manifolds $E^4(n,K)$ which are homeomorphic but not diffeomorphic. In particular, there are infinitely many distinct symplectic 4-manifolds homeomorphic
to $E^4(n)$ \cite{FS} .

In higher dimensions it gives many examples of manifolds which are diffeomorphic but not equivalent as symplectic manifolds.
  Theorem 13 in \cite{f}  implies that the Gromov invariants of $M^4_{\phi}$ are related to symplectic Floer homology  of $\phi $ via Lefschetz  zeta
function  $ L_\phi(t)$. We hope that  the  symplectic zeta function $
 F_\phi(t)$ give rise to a new invariant of symplectic 4-manifolds.

\end{rk}


\begin{thebibliography}{10}

 
\bibitem{AY}
Pierre Arnoux and Jean-Christophe Yoccoz,
 Construction de diff\'eomorphismes pseudo-{A}nosov.
 {\em C. R. Acad. Sci. Paris S\'er. I Math.}, 292(1):75--78, 1981

\bibitem{am}
M. Artin and B. Mazur,
On periodic points,
Annals of Math., 81 (1965), 82-99.
\bibitem{Br}
Jeffrey~F. Brock, Weil-{P}etersson translation distance and volumes of mapping tori.
{\em Comm. Anal. Geom.}, 11(5):987--999, 2003.

\bibitem{c1}
A. Cotton-Clay, Symplectic Floer homology of area-preserving surface diffeomorphisms, Geometry and Topology 13(2009), 2619-2674.
\bibitem{c2} 
A. Cotton-Clay, A sharp bound on fixed points of surface symplectomorphisms in each mapping class, arXiv: 1009.0760[math.SG], 2010.

\bibitem{ds}
S.~Dostoglou and D.~Salamon,  Self dual instantons and holomorphic
  curves, Annals of Math., 139 (1994), 581--640.

\bibitem{flp}
A.~Fathi, F.~Laudenbach, and V.~Po{\'e}naru, Travaux de Thurston sur
  les surfaces, Ast{\'e}risque, vol. 66--67, Soc. Math. France, 1979.

\bibitem{fv}
A. L. Fel'shtyn,
New zeta function in dynamic.
in Tenth Internat. Conf. on Nonlinear Oscillations,
Varna, Abstracts of Papers, Bulgar. Acad. Sci., 1984, 208
\bibitem{fl}
A.L. Fel'shtyn,
New zeta functions for dynamical systems
   and Nielsen fixed point theory.
in : Lecture Notes in Math. 1346, Springer, 1988, 33-55.


\bibitem{FelshB}
A.~Fel'shtyn, \emph{Dynamical zeta functions, {N}ielsen theory and
  {R}eidemeister torsion}, Mem. Amer. Math. Soc. \textbf{147} (2000),
 no.~699,
  xii+146. MR{2001a:37031}

\bibitem{f}
Fel'shtyn A.L., Floer homology, Nielsen theory and symplectic zeta
functions. Proceedings of the Steklov Institute of Mathematics, Moscow, vol. 246,  2004,  pp. 270-282.

\bibitem{ff}
Fel'shtyn A.L., Nielsen theory, Floer homology and a generalisation of the Poincare - Birkhoff theorem. Journal of Fixed Point Theory and Applications, no.2, v.3(2008), 191-214.
\bibitem{fh1}
A.L.Fel'shtyn,R.Hill,
The Reidemeister zeta function with applications to
  Nielsen theory and a connection with Reidemeister torsion. 
  K-theory,v.8,n.4,1994,p.367-393  
\bibitem{fh2}
A.L.Fel'shtyn,R.Hill,
Trace formulae, zeta functions, congruences and Reidemeister torsion in 
Nielsen theory. Forum Mathematicum, v.10,  n.6,  1998,  641-663.  



\bibitem{Floer1}
A.~Floer,
Morse theory for Lagrangian intersections.
J. Differential Geom., 28(1988), 513-547.



\bibitem{FS}
R. Fintushel, R. Stern, Knots, links and 4-manifolds, Invention. Math. 134(1998), 363-400.

\bibitem{han}
M. Handel,
The entropy of orientation reversing homeomorphisms of surfaces.
 Topology 21(1982), 291-296.
 
 \bibitem{HK}
Eriko Hironaka and Eiko Kin,
 A family of pseudo-{A}nosov braids with small dilatation.
 {\em Algebr. Geom. Topol.}, 6:699--738 (electronic), 2006.


\bibitem{G}
R.~Gautschi,
Floer homology of algebraically finite mapping classes
 J. Symplectic Geom. 1(2003), no.4, 715-765.

\bibitem{Gromov}
M.~Gromov,
Pseudoholomorphic curves in symplectic manifolds.
 Invent. Math.82(1985), 307-347.
\bibitem{IP} E.-N. Ionel, Th. Parker, Gromov invariants and symplectic maps,
Math. Ann. 314(1999), 127-158.


\bibitem{I}
N. V. Ivanov,
Entropy and the Nielsen Numbers.
Dokl. Akad. Nauk SSSR 265 (2) (1982), 284-287 (in Russian);
English transl.:
Soviet Math. Dokl. 26 (1982), 63-66.
\bibitem{Iv}
N.~V. Ivanov,
 Coefficients of expansion of pseudo-{A}nosov homeomorphisms.
{\em Zap. Nauchn. Sem. Leningrad. Otdel. Mat. Inst. Steklov. (LOMI)},
  167(Issled. Topol. 6):111--116, 191, 1988.

\bibitem{j}
B. Jiang,
Estimation of the number of periodic orbits.
Pacific Jour. Math., 172(1996), 151-185.

\bibitem{koba}
T. Kobayashi,
Links of homeomorphisms of surfaces and topological entropy.
Proceed. of the Japan Acad. 60(1984), 381-383.


\bibitem {Li} 
Michele Linch,  A comparison of metrics on {T}eichm\"uller space.
 {\em Proc. Amer. Math. Soc.}, 43:349--352, 1974.


\bibitem {m}
A. Manning,
Axiom A diffeomorphisms have rational zeta function.
Bull. London Math. Soc. 3 (1971), 215-220.
 
\bibitem{MS}
D.~McDuff and D.~A. Salamon,
J-holomorphic Curves and Symplectic Topology. AMS Colloquium Publications,
Vol. 52, 2004.


 \bibitem{Mc2}
Curtis~T. McMullen,
Polynomial invariants for fibered 3-manifolds and {T}eichm\"uller
  geodesics for foliations.
{\em Ann. Sci. \'Ecole Norm. Sup. (4)}, 33(4):519--560, 2000.

\bibitem{Mk}
Hiroyuki Minakawa,
Examples of pseudo-{A}nosov homeomorphisms with small dilatations.
{\em J. Math. Sci. Univ. Tokyo}, 13(2):95--111, 2006.


\bibitem{Mo}
J.~Moser,
 On the volume elements on a manifold.
  Trans. Amer. Math. Soc., 120:286--294, 1965.
\bibitem{Pe}
R.~C. Penner,
 Bounds on least dilatations.
 {\em Proc. Amer. Math. Soc.}, 113(2):443--450, 1991.


\bibitem{pf}
V. B. Pilyugina and A. L. Fel'shtyn,
The Nielsen zeta function.
Funktsional. Anal. i Prilozhen. 19 (4) (1985), 61-67 (in Russian);
English transl.:
Functional Anal. Appl. 19 (1985), 300-305.
\bibitem{p} M. Pozniak, Floer homology, Novikov rings and clean intersections, PhD thesis,University of Warwick (1994).
\bibitem{S}
P.~Seidel,
 Symplectic Floer homology and the mapping class group.
  Pacific J. Math. 206(2002), no. 1, 219-229.
\bibitem{S1}
P.~Seidel,
Braids and symplectic four-manifolds with abelian fundamental group.
Turkish  J. Math. 26(2002), no.1, 93-100.
\bibitem{sh}
 M. Shub,
  Dynamical systems, filtrations and entropy, Bull. Amer. Math. Soc. 80(1974), 27-41.

\bibitem{Sm}
Ivan Smith,
Floer cohomology and pencils of quadrics.
arXiv:1006.1099v1[math.SG], 2010.

 \bibitem{su}
 D. Sullivan,
 Travaux de Thurston sur les groupes quasi-Fuchsiens et les varietes hyperboliques de dimension 3 fibres sur $S^1$.
 Seminar Bourbaki 554(1979/80), 1-19.

\bibitem{Th}
W.~P. Thurston,
 On the geometry and dynamics of diffeomorphisms of surfaces.
 Bull. Amer. Math. Soc., 19(2):417--431, 1988.
\bibitem{thu3}
W. Thurston,
Hyperbolic structures on 3-manifolds, II: surface groups and 3-manifolds which fibers over the circle. Preprint.


\end{thebibliography}
\end{document}